\documentclass[10pt, 3p, lefttitle]{elsarticle}
\journal{arXiv}

\usepackage{amssymb}
\usepackage[fleqn]{amsmath}
\usepackage{amsthm}
\usepackage{empheq}
\usepackage{cancel}
\usepackage{xcolor}
\usepackage{nicefrac}
\usepackage[hidelinks]{hyperref}
\usepackage{subcaption}

\newtheorem{theorem}{Theorem}
\theoremstyle{definition}

\newtheorem{remark}{Remark}

\newcommand{\bs}[1]{\boldsymbol{#1}}
\DeclareMathOperator*{\argmax}{argmax}

\begin{document}

\begin{frontmatter}

\title{An Efficient Second-Order Adaptive Procedure for Inserting
CAD Geometries into Hexahedral Meshes using Volume Fractions}
\author[sandia]{Brian N. Granzow\corref{correspondence}}
\author[sandia]{Stephen D. Bond}
\author[sandia]{Michael J. Powell}
\author[sandia]{Daniel A. Ibanez}
\address[sandia]{Sandia National Laboratories \\
P.O. Box 5800 \\
Albuquerque, NM 87185-1321}
\cortext[correspondence]{Corresponding author, bngranz@sandia.gov}
\fntext[label1]{
Sandia National Laboratories is a multimission laboratory managed and operated
by National Technology and Engineering Solutions of Sandia LLC, a wholly owned
subsidiary of Honeywell International Inc. for the U.S. Department of Energy's
National Nuclear Security Administration under contract DE-NA0003525. This
article describes objective technical results and analysis. Any subjective
views or opinions that might be expressed in the article do not necessarily
represent the views of the U.S. Department of Energy or the United States
Government.
}


\begin{abstract}
This paper is concerned with inserting three-dimensional computer-aided design
(CAD) geometries into meshes composed of hexahedral elements using a volume
fraction representation. An adaptive procedure for doing so is presented.
The procedure consists of two steps. The first step performs spatial
acceleration using a $k$-d tree. The second step involves subdividing
individual hexahedra in an adaptive mesh refinement (AMR)-like fashion and
approximating the CAD geometry linearly (as a plane) at the finest subdivision.
The procedure requires only two geometric queries from a CAD kernel:
determining whether or not a queried spatial coordinate is inside or
outside the CAD geometry and determining the closest point on the CAD
geometry's surface from a given spatial coordinate. We prove that the
procedure is second-order accurate for sufficiently smooth geometries and
sufficiently refined background meshes. We demonstrate the expected order
of accuracy is achieved with several verification tests and illustrate the
procedure's effectiveness for several exemplar CAD geometries.
\end{abstract}

\end{frontmatter}

\section{Introduction}
\label{sec:introduction}

Computational meshes are an inherent component of many numerical methods.
Inserting geometric objects into such meshes is often a requisite procedure
during simulation initialization, particularly when considering interactions
between multiple materials
\cite{hirt1981volume, niederhaus2023alegra, weiss2016spatially, rider1998reconstructing,
aftosmis1998robust}.
In this context, the location, orientation, and shape of material interfaces is often encoded
in the form of per-material \emph{volume fractions}. A volume fraction
denotes the ratio of a material's volume within a given mesh element to the
volume of the element itself. This paper is concerned with inserting
three-dimensional geometries represented by computer-aided design (CAD)
software into meshes composed of \emph{hexahedral} elements via volume
fractions.

A common workflow for inserting a CAD geometry into a background
mesh with a volume fraction representation (see e.g. \cite{weiss2016spatially})
involves triangulating the CAD geometry surface, iterating over multiple
sample points within each element in the mesh, and at each sample point,
performing an in/out query with respect to the triangulated surface.
We propose a novel volume fraction insertion
procedure that operates directly on the CAD geometry.
Given a point, we are able to directly query the CAD model for whether
that point is inside or outside the geometry as well as which point on the
CAD surface is closest. This precludes the need to triangulate complex surfaces.
By using the concept of a closest point and a bounding sphere, we are able
to apply a spatial acceleration tree structure to the mesh elements,
producing a method whose runtime scales sublinearly with the number of mesh
elements. 

A key realization is the following: if, given a set of many mesh elements, one
can compute a bounding sphere around them and show that the entire bounding sphere
is inside or outside the geometry, then all of those mesh elements may be marked
as entirely inside or outside the geometry.
We then apply this concept hierarchically following the structure of a $k$-d tree
built from the centroids of the mesh elements. This spatial acceleration
tree structure allows us to avoid querying most of the mesh elements during
the volume fraction computation. Within a single element, we adaptively subdivide
the element in octree fashion and apply the same bounding sphere concept to
guide the subdivision. Finally, for each of the finest subhexahedra, we
approximate the CAD geometry's surface as a plane and compute a volume fraction
contribution using the intersection of this plane and the subhex.

The remainder of this paper is structured as follows. First, we provide
a detailed description of the proposed volume fraction insertion procedure,
as described in three parts: the construction of a $k$-d tree based on a
given set of hexahedral mesh elements, the use of the $k$-d tree to spatially
accelerate volume fraction assignment, and the use of an AMR-like procedure
to obtain accurate volume fraction representations for individual hexahedra
that intersect the CAD geometry's surface. Next, we prove that this procedure
attains second-order accuracy in volume under specific conditions.
Lastly, we present results that demonstrate the expected order of accuracy
and highlight the effectiveness and efficiency of the procedure.

\section{Volume Fraction Insertion Procedure}
\label{sec:procedure}

\subsection{$k$-d Tree Construction}
\label{ssec:kdtree}

Let $\{K^e\}_{e=1}^{n_{\text{el}}}$ denote a set of $n_{\textup{el}}$ hexahedra with
mean coordinate centroids $\{ \bs{c}^e \}_{e=1}^{n_{\text{el}}}$. The volume
fraction insertion procedure begins by constructing a $k$-d tree that spatially
partitions the coordinates $\{\bs{c}^e\}_{e=1}^{n_{\text{el}}}$ using recursive
coordinate bisection \cite{bentley1975multidimensional, cormen2022introduction}.
Briefly, the $k$-d tree we use is a binary tree that separates
coordinates in $d$-dimensional ($d=3$ presently) space with axis-aligned
hyperplanes. For a given non-leaf node in the tree, recursive coordinate
bisection chooses the axis normal to this separating plane as
\begin{equation}
l = \argmax_{k} \text{diam} ( \{ \bs{e}_k \cdot \bs{c}^n \}_{n=1}^{n_{\text{node}}} ),
\end{equation}
where $\bs{e}_k$ denotes the $k^{\text{th}}$ basis vector and
$n_{\text{node}}$ denotes the number of centroids associated with the
current node. The node's two children contain 
$\lfloor \frac12 n_{\text{node}} \rfloor$ or
$\lfloor \frac12 n_{\text{node}} \rfloor +1$
centroids with $l$-coordinate greater than or less than the
separating plane's $l$-coordinate, given as:
\begin{equation}
x_{l} = \frac{1}{n_{\text{node}}} \sum_{n=1}^{n_{\text{node}}} c^n_l.
\end{equation}
A bounding sphere that fully covers the set of hexahedra associated with
a given node is computed and stored for each node in the tree. We
conceptually illustrate this $k$-d tree in two dimensions
in Figure \ref{fig:tree}. The construction of the $k$-d tree occurs once
before any geometry is inserted into the mesh and can be reused
when inserting multiple geometries.
\begin{figure}[ht!]
\centering
\includegraphics[width=.5\linewidth]{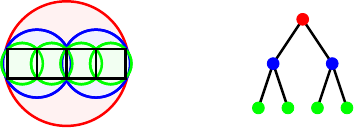}
\caption{Left: A mesh with four quadrilateral elements (black) and
bounding circles associated with nodes in the $k$-d tree (red, blue, green).
Right: The $k$-d tree structure associated with the mesh on the left, where
the root node is associated with all mesh elements and leaf nodes are 
associated with a single mesh element.}
\label{fig:tree}
\end{figure}

\subsection{Bulk Volume Fraction Sampling}
\label{ssec:bulk_sampling}

Let $\Omega \subset \mathbb{R}^3$ denote a bounded domain represented
by a CAD kernel with boundary $\partial \Omega$. Spatial acceleration
for representing the domain $\Omega$ in the mesh elements
$\{ K^e \}_{e=1}^{n_{\text{el}}}$ by volume fractions
is achieved by recursively descending the nodes of
the $k$-d tree. A given node's
bounding sphere is determined to be fully inside, fully outside, or
intersecting the geometry $\Omega$ using two CAD queries, as
discussed in the next
paragraph. If the bounding sphere is found to be either fully inside
or fully outside of $\Omega$, then all of the tree node's
hexahedra are assigned volume fractions of one or zero, respectively,
and recursion is terminated. If the bounding sphere is found to be
intersecting $\Omega$, then recursion continues. If the recursion procedure
reaches a leaf node, then recursion is terminated and an individual
volume fraction is assigned using the AMR-like procedure described in
Section \ref{ssec:element_sampling}.

\begin{figure}[ht!]
\centering
\includegraphics[width=.35\linewidth]{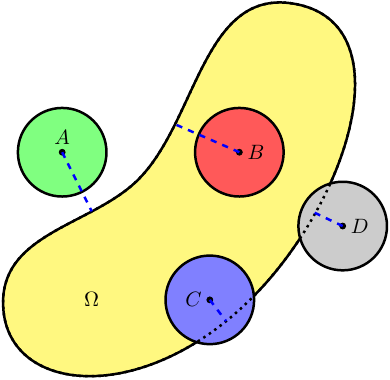}
\caption{Four potential bounding sphere classifications when
comparing the sphere's radius to the closest distance to the geometry's
surface. (A) The bounding sphere is fully outside of $\Omega$. (B)
The bounding sphere is fully inside of $\Omega$. (C) The bounding
sphere intersects $\Omega$ and its center $\bs{x}^C_c \in \Omega$.
(D) The bounding sphere intersects $\Omega$ and its center
$\bs{x}^D_c \notin \Omega$.}
\label{fig:bulk}
\end{figure}

To classify a bounding sphere as fully inside, fully outside, or
intersecting the geometry $\Omega$, the bounding sphere's center
$\bs{x}^{\text{sphere}}_c$ is first determined to be inside or outside of $\Omega$
using a CAD in/out query. Then the point $\bs{x}_p$ on the geometry's
surface that is closest to the bounding sphere's center $\bs{x}^{\text{sphere}}_c$
is determined using a CAD closest point query. The Euclidean
distance $d = \| \bs{x}^{\text{sphere}}_c - \bs{x}_p \|$ from this point to
the sphere's center is computed
and compared to the bounding sphere's radius $r$.
This leads to four potential scenarios as illustrated by Figure
\ref{fig:bulk}. If $d > r$ and $\bs{x}^{\text{sphere}}_c \notin \Omega$, then the bounding
sphere is fully outside $\Omega$ (sphere A), if $d > r$ and
$\bs{x}^{\text{sphere}}_c \in \Omega$, then the bounding sphere is
fully inside $\Omega$ (sphere B), and if $d < r$, then the bounding
sphere intersects $\Omega$ (spheres C and D).

\begin{remark}
Floating point precision and other factors in CAD kernel
implementations can lead to `non-watertight' geometries, for which
in/out CAD queries may not be robust. Strategies exist
\cite{jacobson2013robust, spainhour2024robust} to generalize the
notion of containment within a body that we do not presently consider
but could be pursued in future work.
\end{remark}

\begin{remark}
The depth of the $k$-d tree can increase with increased topological
complexity; for example, high-genus geometries and/or additively
manufactured parts with complex geometric features. This can lead to a
loss of spatial acceleration because the $k$-d tree nodes will not
admit large subsets of elements that are entirely inside or outside
of the CAD geometry. The overall accuracy of the procedure,
will not be affected, only the spatial acceleration performance.
\end{remark}

\subsection{Adaptive Element Volume Fraction Sampling}
\label{ssec:element_sampling}

Once a leaf in the $k$-d tree is reached, the volume of intersection
between the leaf's corresponding mesh hexahedron and the domain $\Omega$
must be determined. This volume is approximated using an octree, where
hexes are recursively subdivided into eight subhexes in an AMR-like fashion.
The left image in Figure \ref{fig:amr} illustrates this subdivision
process in two spatial dimensions.
Presently, the number of AMR subdivisions is a user-specified input,
where a larger number of subdivisions will correspond to a more accurate
volume fraction representation of the domain $\Omega$.
It is, however, difficult for a user to know \emph{a priori} the number
of AMR subdivisions required to achieve a desired tolerance. As
an alternative, one could envision using the asymptotic error bound
(described in Section \ref{sec:error_estimate}) to adaptively determine
a local (per-element) termination criteria based on a user-specified
volume error tolerance.
We leave this as
an avenue for future research. To determine if an individual subhex
is fully inside, fully outside, or intersecting the CAD geometry's
surface, a bounding sphere of the subhex is constructed centered
at the subhex's mean coordinate centroid $\bs{x}^{\text{sphere}}_c$
and the same logic described in the previous section is
applied (where a CAD in/out query and a CAD closest point query
fully characterizes the subhex's in/out/intersecting status). Subhexes that are fully
outside or fully inside of the domain $\Omega$ are given subhex
volume fractions of zero and one, respectively, and recursion
is terminated.

\begin{figure}[ht!]
\centering
\begin{subfigure}{0.2\textwidth}
\centering
\includegraphics[width=.99\linewidth]{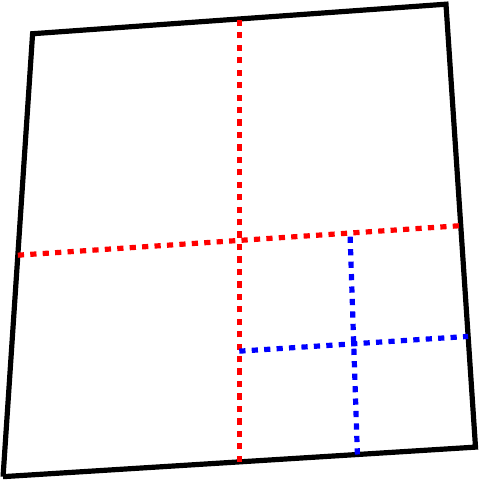}
\end{subfigure}%
\hspace{4em}
\begin{subfigure}{0.2\textwidth}
\centering
\includegraphics[width=.99\linewidth]{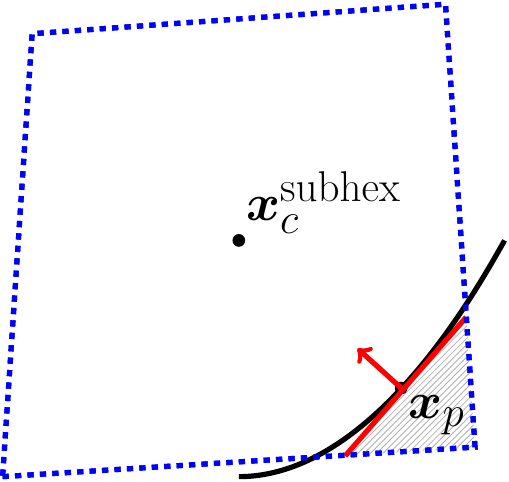}
\end{subfigure}
\caption{Left: A two-dimensional representation of a hexahedron
(solid black) subdivided to two AMR levels. Right: A
two-dimensional representation of the intersection of a subhex
(dashed blue) with a CAD geometry (solid black line), the planar
approximation of the geometry's surface within the subhex
(solid red line), and the plane's unit normal (solid red line
with arrow). The shaded subvolume corresponds to the approximated
CAD geometry within the subhex.}
\label{fig:amr}
\end{figure}

At subhexes that have reached the maximal AMR depth, the
geometry is approximated by a plane, as illustrated by
the right image in Figure \ref{fig:amr}. The plane is defined
by a point $\bs{x}_p$ on the CAD surface and a plane surface
normal $\bs{n}_p$. The point $\bs{x}_p$ is chosen as the the
closest point on the geometry's surface to the the subhex's
mean coordinate centroid $\bs{x}^{\text{subhex}}_c$. The plane
surface normal $\bs{n}_p$ is
chosen to be parallel to the vector
$\bs{x}^{\text{subhex}}_c -\bs{x}_p$, which is also,
by definition,
parallel to the geometry's outward surface normal
(as long as the domain $K^{\text{subhex}} \cap \Omega$ is simply
connected and $\partial \Omega$ is sufficiently smooth).
The normal vector is oriented outward to the surface
$\partial \Omega$ by querying whether the subhex
centroid is inside or outside of the domain $\Omega$.
Using \texttt{r3d} \cite{powell2015exact}, the subhex is
cut by the constructed plane and the volume $V^{\text{poly}}$
of the resultant polyhedron is computed. Lastly, the
volume fraction $V^{\text{poly}} / \text{meas}(K^e)$ is
contributed to the original hex's total volume fraction.
In the degenerate case when the closest point $\bs{x}_p$ is
identical to the subhex centroid $\bs{x}^{\text{subhex}}_c$,
we choose a subhex volume fraction of $\frac12$.

\begin{remark}
If the domain $K^{\text{subhex}} \cap \Omega$ is not simply
connected or contains more than a single surface in the interior
of $K^{\text{subhex}}$ (such as a small filament), then it may
be possible to employ more sophisticated techniques
\cite{dyadechko2008reconstruction, ahn2009adaptive, jemison2015filament}
to capture and represent the geometry within the mesh.
In particular, we propose moment-of-fluid initialization from CAD
geometry as an avenue for future research.
\end{remark}

\section{Error Estimate}
\label{sec:error_estimate}

\begin{theorem}
For a bounded CAD geometry $\Omega \subset \mathbb{R}^d$ with a smoothly
twice-differentiable boundary, $\partial \Omega$, and a sufficiently
refined background mesh, the volume insertion error is bounded
by $|V_{\textup{exact}} - V_h | \leq C h^{2}$, where $V_h$ and
$V_{\textup{exact}}$ are the approximate and exact volumes,
respectively, $C$ is a constant, and $h$ is the background element
size.
\end{theorem}

\begin{proof}

Let $f(\bs{x}, \Omega)$ denote a signed distance function, defined as
\[
f(x,\Omega) :=
\begin{cases}
\begin{aligned}
-&d(\bs{x}, \partial \Omega) && \text{if} \; \bs{x} \in \Omega, \\
&d(\bs{x}, \partial \Omega) && \text{otherwise},
\end{aligned}
\end{cases}
\]
where $d(\bs{x}, \partial \Omega)$ denotes the minimal distance
from a point $\bs{x}$ to a surface $\partial \Omega$, given as
\[
d(\bs{x}, \partial \Omega) := \inf_{\bs{y} \in \partial \Omega}
\| \bs{x} - \bs{y} \|_2.
\]
Note the signed distance $f$ is positive for points exterior to $\Omega$
and negative for points interior to $\Omega$.

Suppose that the CAD geometry is covered by a mesh $\mathcal{T} := \cup_e \overline{K^e}$,
with elements $\left\{ K^e \right\}_{e=1}^{n_{\text{el}}}$ of size $h$.
Splitting over elements, the exact volume is
\[
V_{\text{exact}} = \sum_{K^e \in \mathcal{T}}
\int_{K^e} \, \Phi(f(\bs{x}, \Omega)) \, \text{d} \bs{x},
\]
where $\Phi$ is the step function
\[
\Phi(x) =
\begin{cases}
\begin{aligned}
&1, && x \leq 0, \\
&0, && \text{otherwise}.
\end{aligned}
\end{cases}
\]
Let $g^e$ be the local linear approximation to $f$ in element $K^e$.
Then the approximate volume represented by the volume fraction insertion
procedure is
\[
V_h = \sum_{K^e \in \mathcal{T}}
\int_{K^e} \, \Phi(g^e(\bs{x})) \, \text{d} \bs{x}.
\]
The total volume insertion error can then be bounded above as
\[
\begin{aligned}
|V_{\text{exact}} - V_h| &=
\left|
\sum_{K^e \in \mathcal{T}}
\left(
\int_{K^e} \Phi(f(\bs{x}, \Omega)) \, \text{d} \bs{x}
- \int_{K^e} \Phi(g^e(\bs{x})) \, \text{d} \bs{x}
\right)
\right| \\
& \leq
\sum_{K^e \in \mathcal{T}} \int_{K^e}
\left|
\Phi(f(\bs{x}, \Omega)) - \Phi(g^e(\bs{x}))
\right| \, \text{d} \bs{x}.
\end{aligned}
\]

For a given $\bs{x}_e \in K^e$, let $\bs{x}_p \in \partial \Omega$
be the closest point on the CAD geometry surface. Assuming the surface
is smoothly differentiable, the first-order constrained optimality
condition requires that the gradient of the signed distance function
is parallel to the surface normal with
\[
\alpha (\bs{x}_p - \bs{x}_e) = \nabla_{\bs{x}} f(\bs{x}_p, \Omega) = \bs{n}_p,
\]
for some scalar $\alpha$, where $\bs{n}_p$ is the outward normal to
$\Omega$ at point $\bs{x}_p$. Using the Taylor series expansion
of $f$ at $\bs{x}_p$, we find
\[
f(\bs{x}, \Omega) =
f(\bs{x}_p, \Omega) +
(\bs{x} - \bs{x}_p) \cdot \nabla_{\bs{x}} f(\bs{x}_p, \Omega) +
\frac12 (\bs{x} - \bs{x}_p)^T H_f(\bs{\xi}, \Omega) (\bs{x} - \bs{x}_p),
\]
where $H_f$ is the Hessian of the signed distance function and $\bs{\xi}$
is some point on the chord connecting $\bs{x}$ and $\bs{x}_p$. Note that
the first two terms of the expansion are exactly the local linear
approximation to $f$ used by the CAD insertion algorithm
\[
g^e(\bs{x}) = f(\bs{x}_p, \Omega) + (\bs{x} - \bs{x}_p) \cdot
\nabla_{\bs{x}} f(\bs{x}_p, \Omega).
\]
In an element of size $h$, one can show that
\begin{equation}
|f(\bs{x}, \Omega) - g^e(\bs{x})| =
\left|
\frac12 (\bs{x} - \bs{x}_p)^T H_f(\bs{\xi}, \Omega)
(\bs{x} - \bs{x}_p)
\right|
\leq \tilde{C} h^2,
\label{eq:proof1}
\end{equation}
where the constant $\tilde{C}$ is proportional to the maximum
spectral radius of $H_f$ over the element.

Introduce the functional
\[
\Psi(s) := \Phi(s f(\bs{x}, \Omega) + (1-s) g^e(\bs{x})),
\]
and note that $\Psi(1) = \Phi(f(\bs{x}, \Omega))$ and
$\Psi(0) = \Phi(g^e(\bs{x}))$. Applying the mean-value theorem
to $\Psi$, we obtain
\[
\int_{K^e} \left| \Phi(f(\bs{x}, \Omega)) - \Phi(g^e(\bs{x})) \right| \, \text{d} \bs{x} =
\int_{K^e} \left| (f(\bs{x}, \Omega) - g^e(\bs{x})) \right|
\delta(\gamma f(\bs{x}, \Omega) + (1-\gamma) g^e(\bs{x})) \, \text{d} \bs{x},
\]
where $\gamma$ is a scalar on the interval in $[0,1]$
and $\delta$ denotes the Dirac delta function.
Using the inequality \eqref{eq:proof1} yields
\[
\int_{K^e}
\left|
\Phi(f(\bs{x}, \Omega)) - \Phi(g^e(\bs{x}))
\right|
\leq
\tilde{C} h^2 \int_{K^e}
\delta(\gamma f(\bs{x}, \Omega) + (1-\gamma) g^e(\bs{x})) \, \text{d} \bs{x}
\leq
\tilde{C} h^2 \hat{C} h^{d-1},
\]
where $\hat{C} h^{d-1}$ is the area of the surface
$\left\{ \bs{x} \in K^e \subset \mathbb{R}^d | \,
(1-\gamma)f(\bs{x}, \Omega) + \gamma g^e(\bs{x}) = 0 \right\}$.
Finally, for a regular CAD surface,
the number of elements intersecting $\partial \Omega$ is
proportional to $1/h^{d-1}$, so that the total volume insertion
error is bounded by
\[
|V_{\text{exact}} - V_h| \leq
\sum_{K^e \in \mathcal{T}} \tilde{C} h^2 \hat{C} h^{d-1}
= Ch^2,
\]
for some constant $C$.
\end{proof}

\section{Results}
\label{sec:results}

In this section, we investigate the efficacy of the proposed volume
fraction insertion procedure for a variety of numerical simulations.
Throughout, we make reference and comparison to a uniform sampling
procedure, which we describe as follows. We consider a reference hexahedron
$K^R := \{ \bs{\xi} : \bs{\xi} \in [-1,1]^3 \}$ and take the coordinates
$\{ \bs{\xi} \}_{i=1}^{n_s}$ defined as the centroids
of the subhexes found by uniformly refining the reference hex
$n_{\text{sub}}$ times, so that the total number of sample points
is given as $n_s = (2^{n_{\text{sub}}})^3$.
Each reference coordinate $\bs{\xi}_i$ is
mapped to a physical coordinate $\bs{x}_i$ using trilinear Lagrange shape
functions and a CAD in/out query is performed at this physical coordinate.
An element volume fraction is then computed as a Jacobian determinant weighted ratio of
the number of sample points found inside the domain $\Omega$ to the total number of
sample points $n_s$.

\subsection{Verification}
\label{ssec:results_verification}

In this section, we demonstrate the proposed volume fraction
procedure achieves the theoretical rate of convergence for two
simple geometries, a unit sphere and a hemispherically capped oblique
cylinder, inserted into two example
background meshes, an axis-aligned hexahedral mesh and the same
mesh with a sinusoidal perturbation applied to its coordinates.
The axis-aligned mesh is defined by the domain
$\{ \bs{x} : \bs{x} \in [-2,2]^3 \}$ with $32^3$ cubical
elements. The sinusoidally perturbed mesh applies the
coordinate transformations
\begin{equation}
\begin{aligned}
x &= x + 0.1\sin(\nicefrac12 \; \pi yz), \\
y &= y + 0.1\sin(\nicefrac12 \; \pi xz), \\
z &= z + 0.1\sin(\nicefrac12 \; \pi xy). \\
\end{aligned}
\label{eq:transform}
\end{equation}
to each vertex of the axis-aligned mesh. The unit sphere is centered at the
origin. The hemispherically capped cylinder is defined as all
points within a radius $r=0.2$ of the line segment defined
by the two end points $[\text{-}0.51,\text{-}0.49,\text{-}0.52]$ and
$[ 0.49, 0.51, 0.48]$. For context, Figure \ref{fig:verify_geoms}
illustrates the two considered geometries, as represented by
inserted volume fractions, in the sinusoidally perturbed mesh.
Additionally, Figures \ref{fig:sphere_elem} and \ref{fig:pill_elem}
illustrate AMR subdivisions for a representative hexahedron in the
sinusoidally perturbed mesh when inserting the sphere and
hemispherically capped cylinder, respectively, with $n_{\text{sub}} = 5$
subdivisions.

\begin{figure}[ht!]
\centering
\begin{subfigure}{0.5\textwidth}
\centering
\includegraphics[width=.99\linewidth]{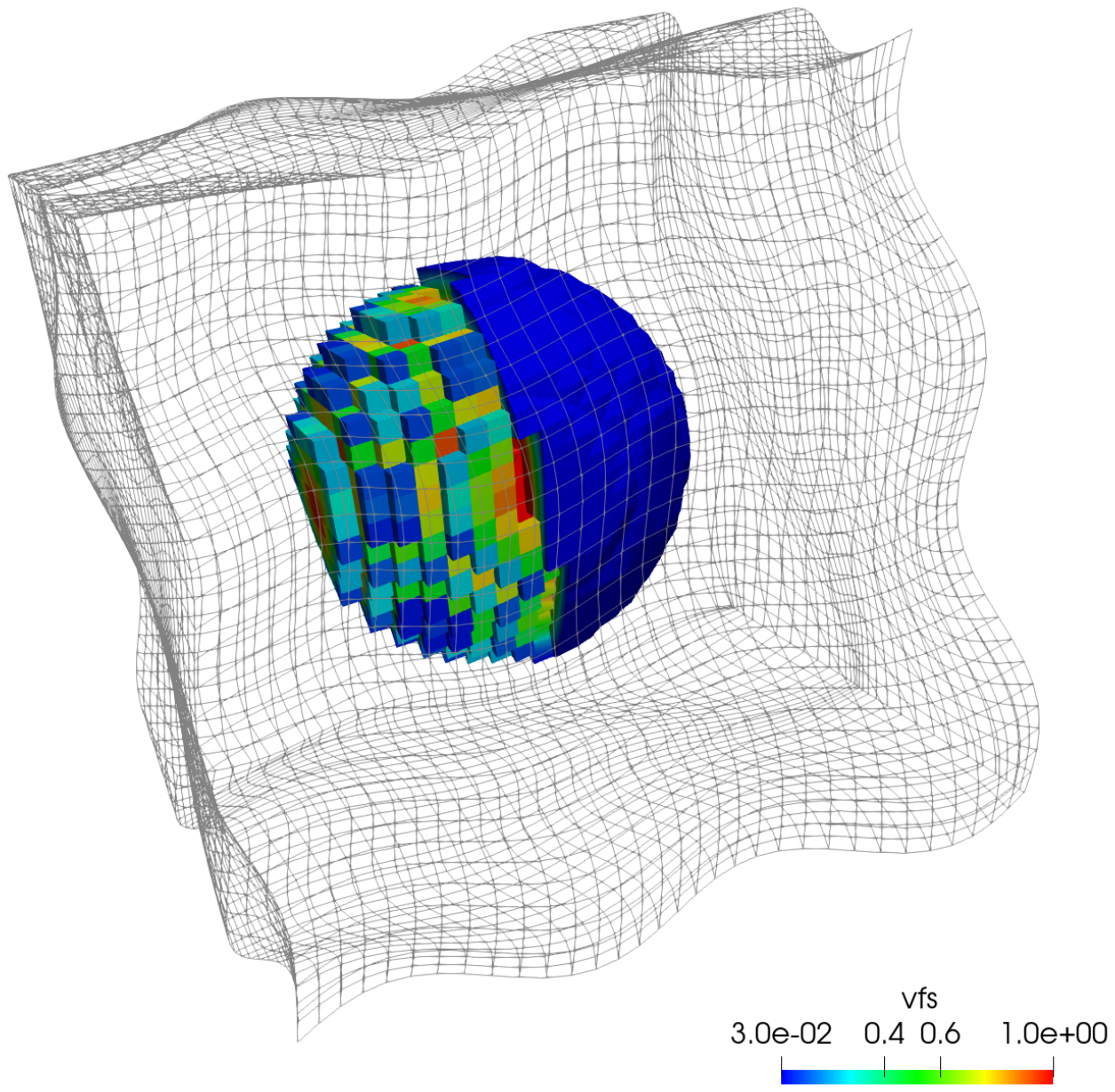}
\caption{Unit sphere.}
\end{subfigure}%
\begin{subfigure}{0.5\textwidth}
\centering
\includegraphics[width=.99\linewidth]{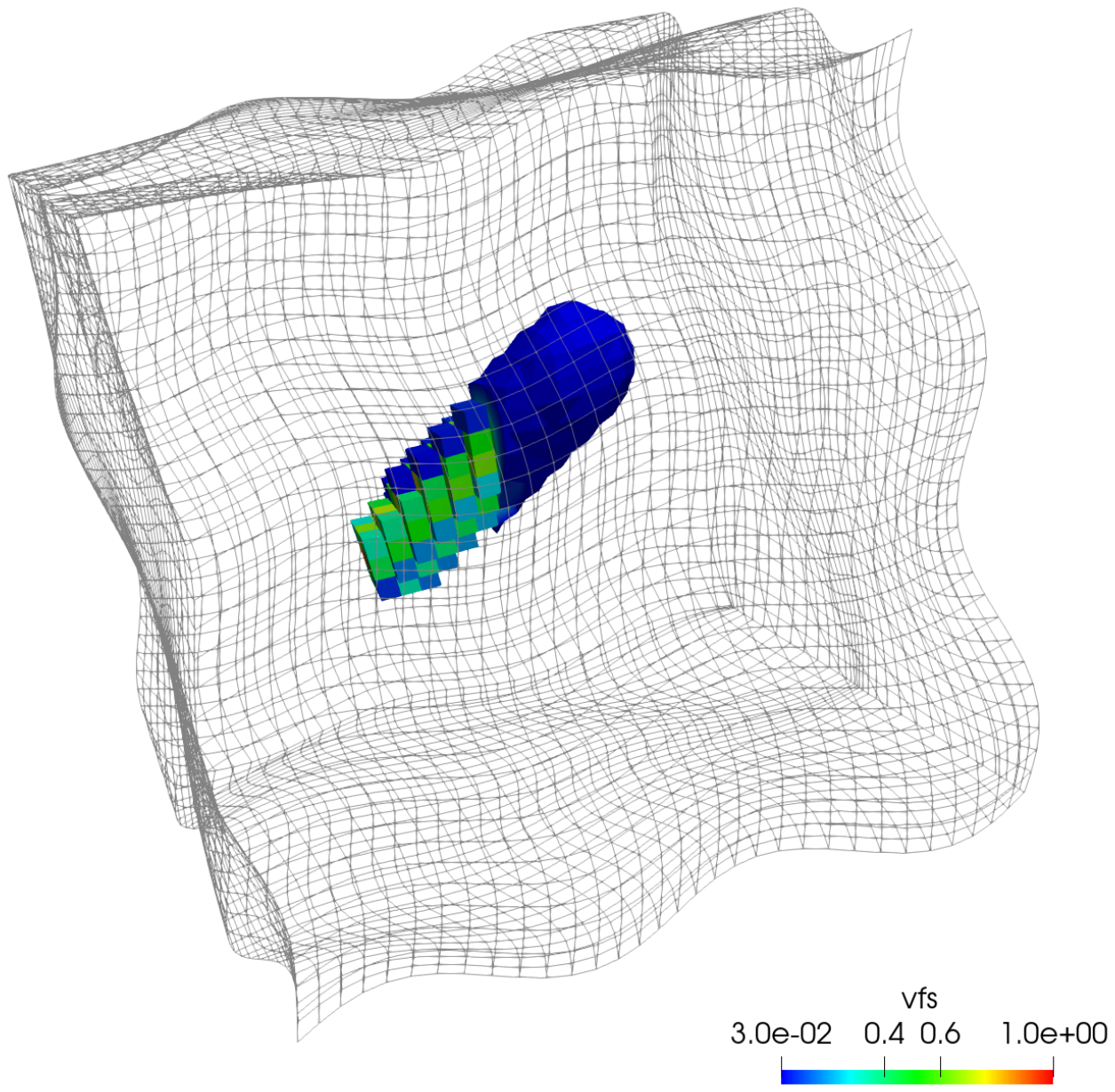}
\caption{Hemispherically capped cylinder.}
\end{subfigure}
\caption{Representations of the verification geometries
inserted into a sinusoidally perturbed mesh where the left portion
of the domain is rendered with cell volume fractions and the right
portion of the domain corresponds to a recovered isosurface.}
\label{fig:verify_geoms}
\end{figure}

\begin{figure}[ht!]
\centering
\begin{subfigure}{0.42\textwidth}
\centering
\includegraphics[width=.99\linewidth]{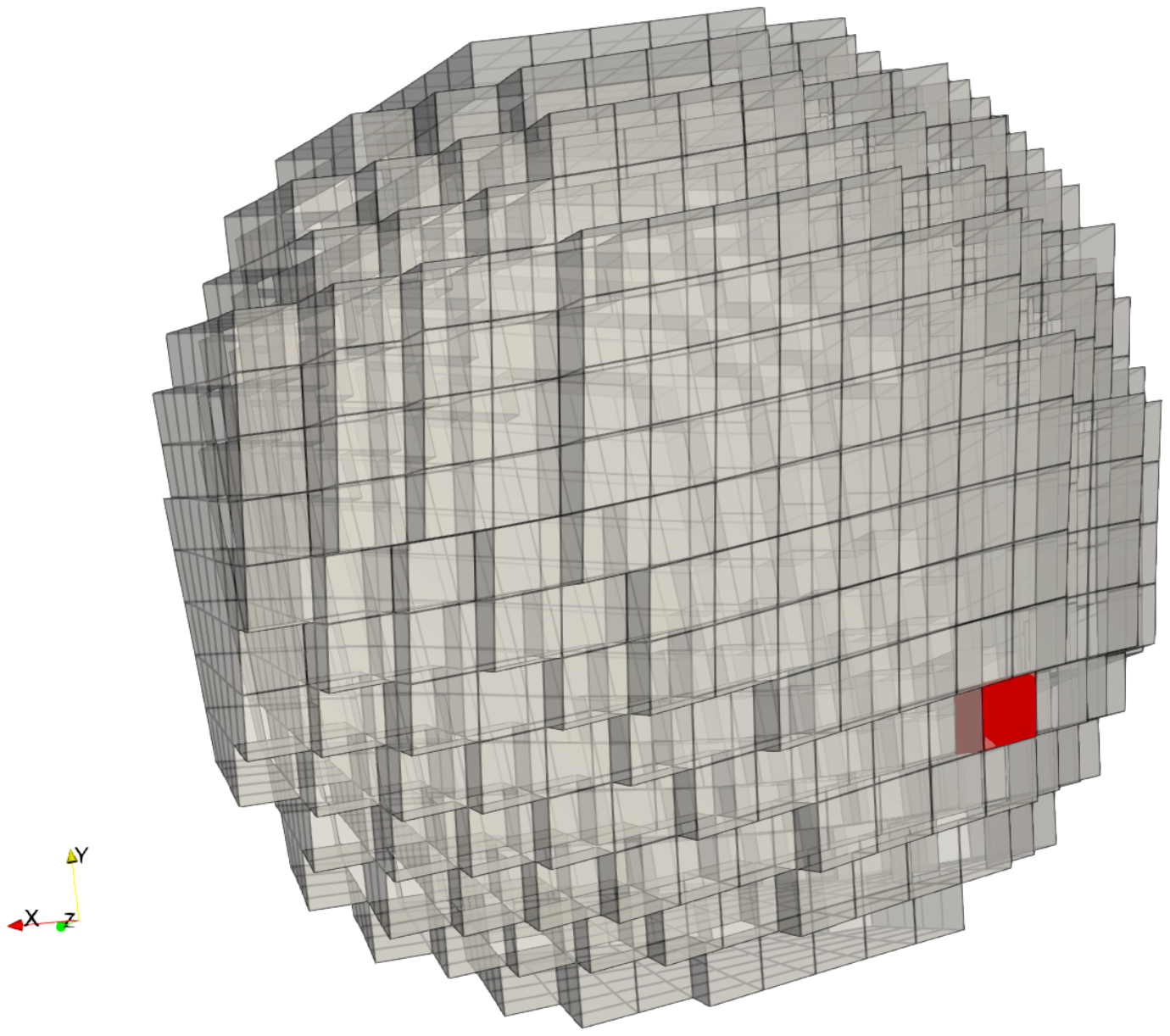}
\caption{Sample element (red).}
\end{subfigure}%
\begin{subfigure}{0.42\textwidth}
\centering
\includegraphics[width=.99\linewidth]{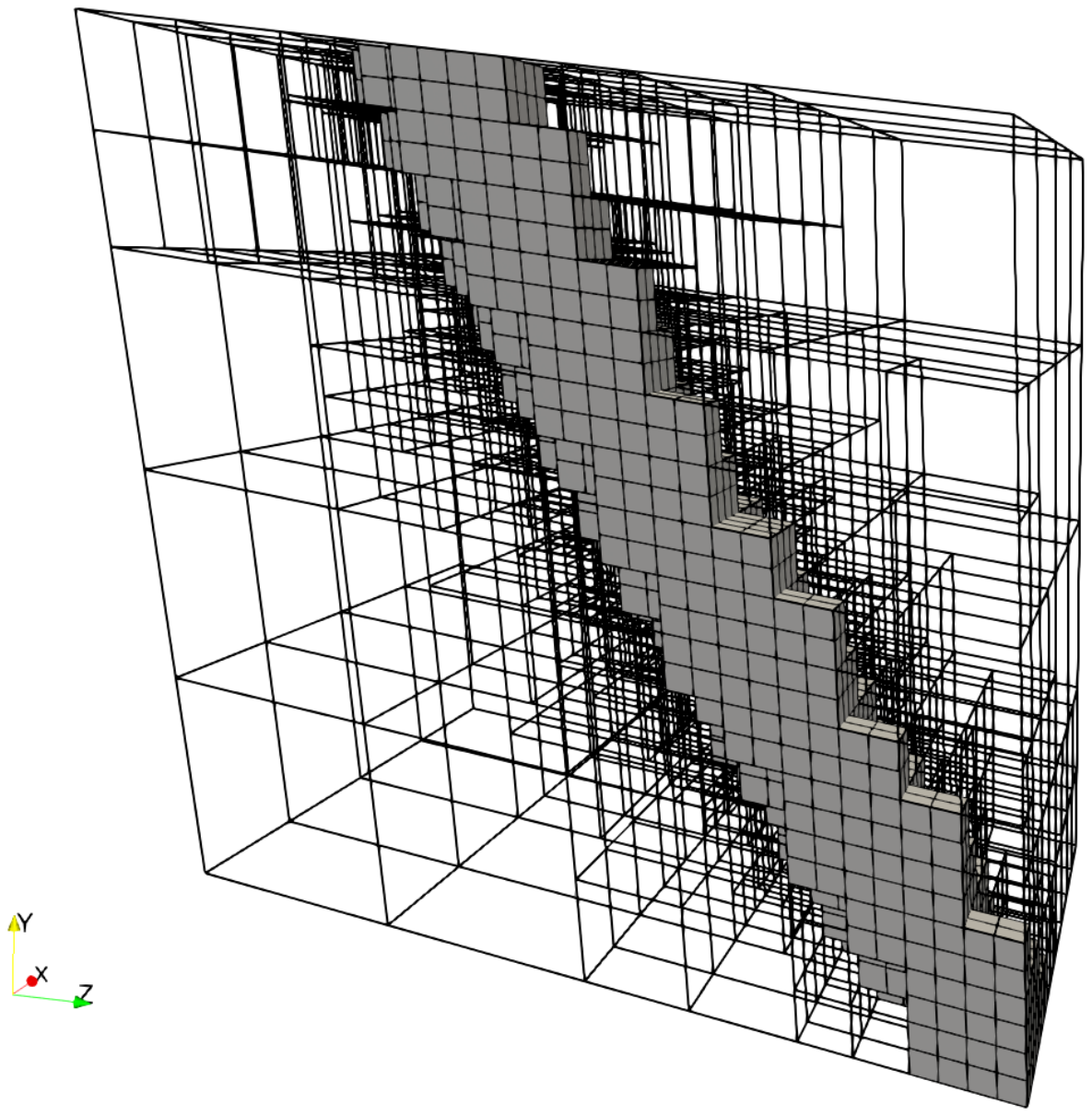}
\caption{Subhexahedra (gray) at the finest AMR subdivision.}
\end{subfigure}
\caption{An illustration of the AMR subdivision procedure for a
chosen element in the case of a sphere inserted into the
sinusoidally perturbed mesh with $n_{\text{sub}}=5$ subdivisions.
Note that the element has been rotated to more clearly illustrate the
AMR structure.}
\label{fig:sphere_elem}
\end{figure}

\begin{figure}[ht!]
\centering
\begin{subfigure}{0.42\textwidth}
\centering
\includegraphics[width=.99\linewidth]{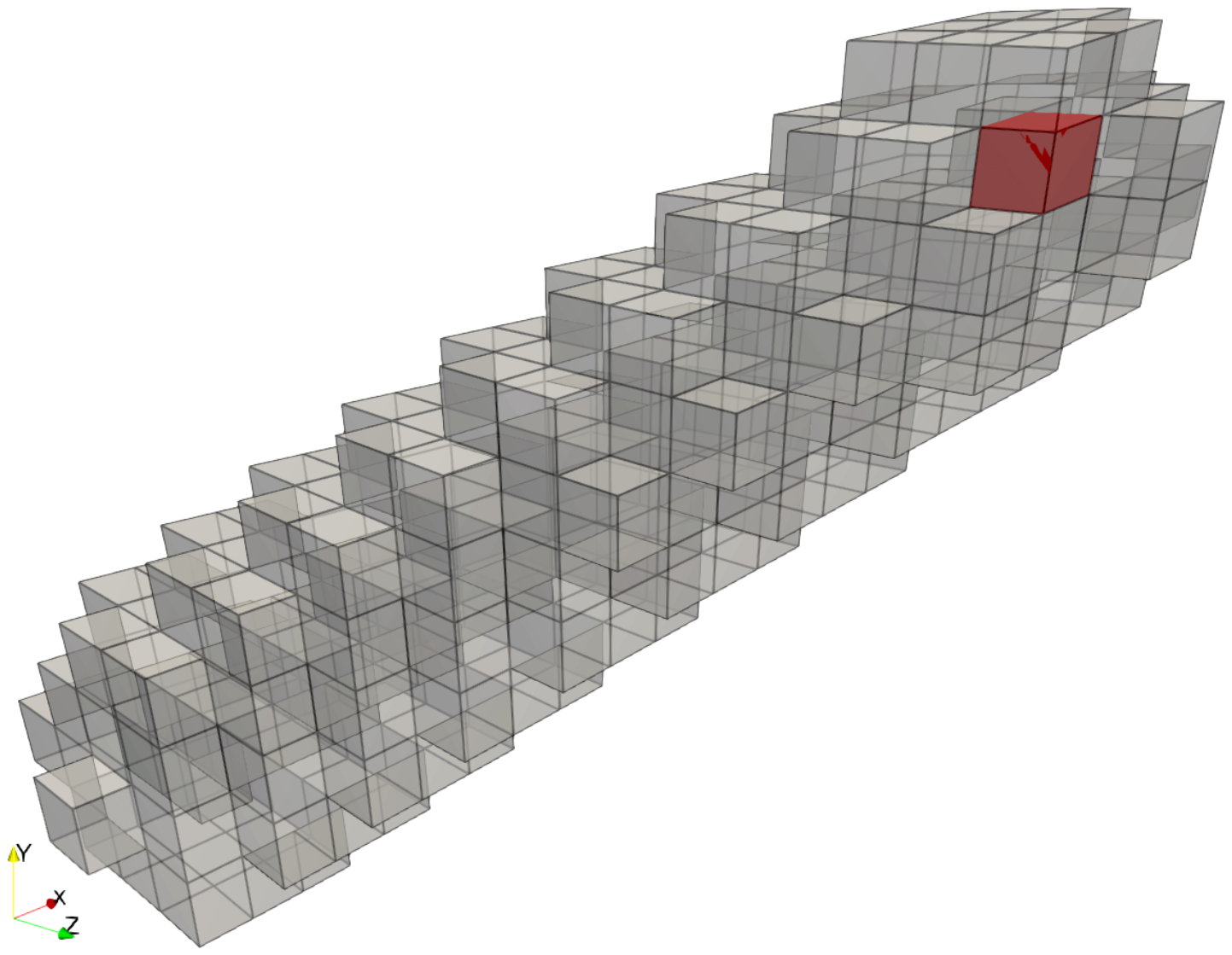}
\caption{Sample element (red).}
\end{subfigure}%
\begin{subfigure}{0.42\textwidth}
\centering
\includegraphics[width=.99\linewidth]{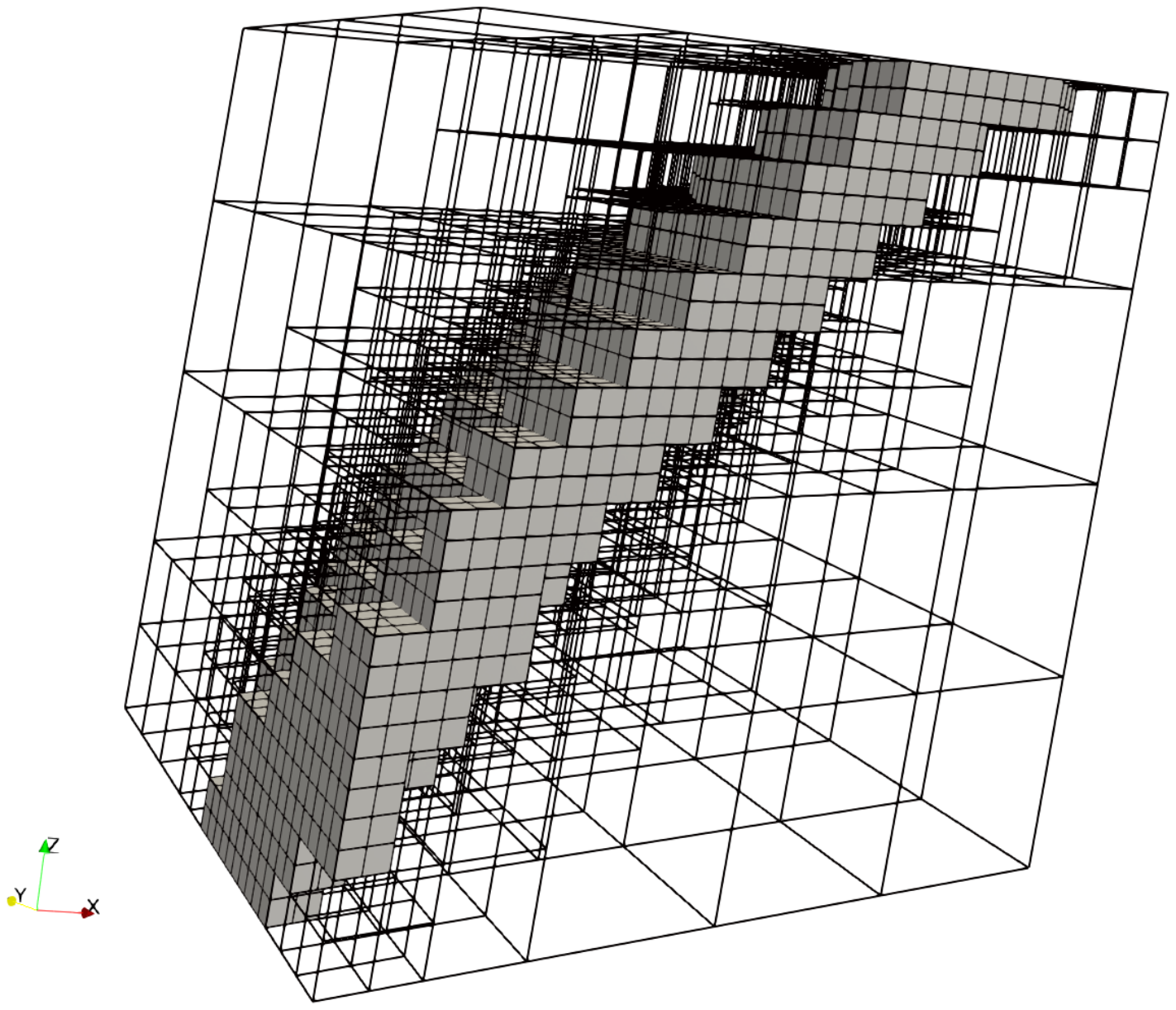}
\caption{Subhexahedra (gray) at the finest AMR subdivision.}
\end{subfigure}
\caption{An illustration of the AMR subdivision procedure for a
chosen element in the case of the spherically capped cylinder
inserted into the sinusoidally perturbed mesh with
$n_\text{sub}=5$ subdivisions. Note that the element has
been rotated to more clearly illustrate the AMR structure.}
\label{fig:pill_elem}
\end{figure}

The left-most plots of Figures \ref{fig:sphere_plots} and
\ref{fig:pill_plots} demonstrate the convergence behavior for
each geometry inserted into each mesh using either the proposed
spatially-accelerated AMR volume fraction insertion procedure or
the simpler uniform sampling approach. For each approach, we measure
the relative error, $\nicefrac{|V_{\text{exact}} - V_h|}{V_{\text{exact}}}$,
when using a maximum of $n_{\text{sub}} = \{0,1,2,3,4,5\}$ hex subdivisions,
or correspondingly $n_s = \{1,8,64,512,4096,32768\}$ sampling points
per element for the uniform sampling approach.
In each of these four plots, we observe that the AMR plane sampling
procedure achieves the theoretical second-order convergence rate.
Additionally, the relative error for both geometries inserted into the
axis-aligned mesh is lower for the AMR plane sampling
approach when compared to uniform sampling at all subdivision numbers.
The left-most plots of Figure \ref{fig:sphere_plots}
demonstrates that the uniform sampling procedure does not always
monotonically decrease the relative error when increasing the
number of sampling points. This can be explained by the fact that
the uniform sampling procedure is essentially approximating the
integral of a discontinuous function using composite midpoint
quadrature, which is not guaranteed to converge at a given rate.
This can be problematic when trying
to achieve greater accuracy in a volume representation, as it may
necessitate costly trial and error samplings with an unknown number
of uniform subdivisions to achieve a desired error.

\begin{figure}[ht!]
\centering
\includegraphics[width=0.7\linewidth]{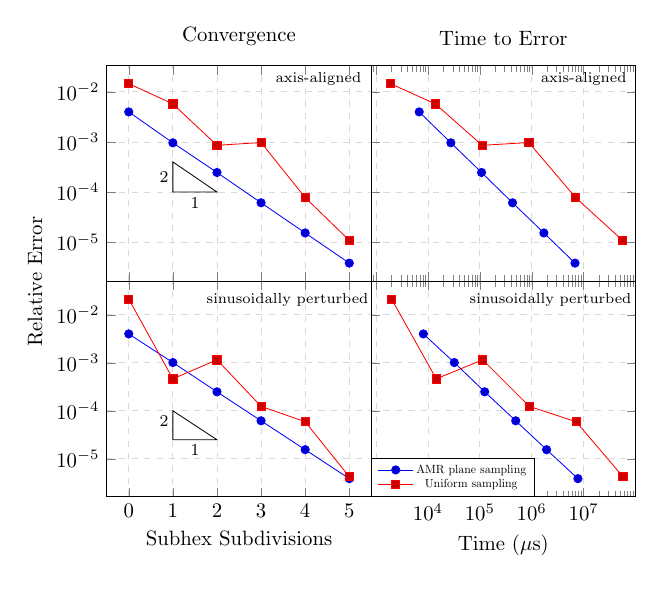}
\caption{Convergence history (left) and time to solution (right)
for a sphere inserted into an axis-aligned mesh (top) and a
sinusoidally perturbed mesh (bottom) for the proposed volume
fraction insertion procedure compared to a uniform sampling
procedure.}
\label{fig:sphere_plots}
\end{figure}

The bottom-left plot of Figure \ref{fig:sphere_plots} shows that
the uniform sampling approach is more accurate at $1$ subdivision
level than the AMR plane sampling approach and the bottom-left
plot of Figure \ref{fig:pill_plots} shows that the uniform sampling
approach is more accurate at nearly all measured subdivision levels.
However, this perceived lack of performance of the AMR plane sampling
approach can be explained when considering computational complexity
and time-to-error of each volume fraction sampling approach.
The uniform sampling procedure scales as
$\mathcal{O}(n_{\text{sub}}^3)$ while the AMR plane sampling approach
scales as $\mathcal{O}(n_{\text{sub}}^2)$ since it refines locally
to surface features.

The right-most plots of Figures \ref{fig:sphere_plots} and
\ref{fig:pill_plots} demonstrate the amount of computational time
needed to achieve a given relative volume insertion error.
In all scenarios, the right-most data point for the AMR plane sampling
approach achieves a more accurate volume insertion in less amount
of time when considering a constant error. This efficiency
is most pronounced in upper-right plot of Figure \ref{fig:pill_plots}
(corresponding to the hemispherically capped cylinder inserted into
the axis-aligned mesh), where the time-to-error for the AMR plane
sampling procedure is multiple orders of magnitude faster than
uniform sampling. This efficacy is least pronounced in
the bottom-right plot of Figure \ref{fig:pill_plots}
(corresponding to the hemispherically capped cylinder inserted into
the sinusoidally perturbed mesh), where the time-to-error for the
AMR plane sampling procedure is slightly less than an order of
magnitude faster than uniform sampling.

The $x$-axis distance between data points in the
right-most plots of Figures \ref{fig:sphere_plots} and
\ref{fig:pill_plots} illustrates that each successive
subdivision increases the time to solution by
about a factor of about $8$ for the uniform sampling approach,
whereas this increase is only about a factor of $4$ for the
AMR plane sampling approach. This agrees with the computational
complexity of the two approaches, as discussed previously.
These plots illustrate that uniform sampling can be a valid
and cost-effective approach if higher volumetric insertion
errors are acceptable, but may become cost-prohibitive
as errors are driven to lower and lower values.

\begin{figure}[ht!]
\centering
\includegraphics[width=0.7\linewidth]{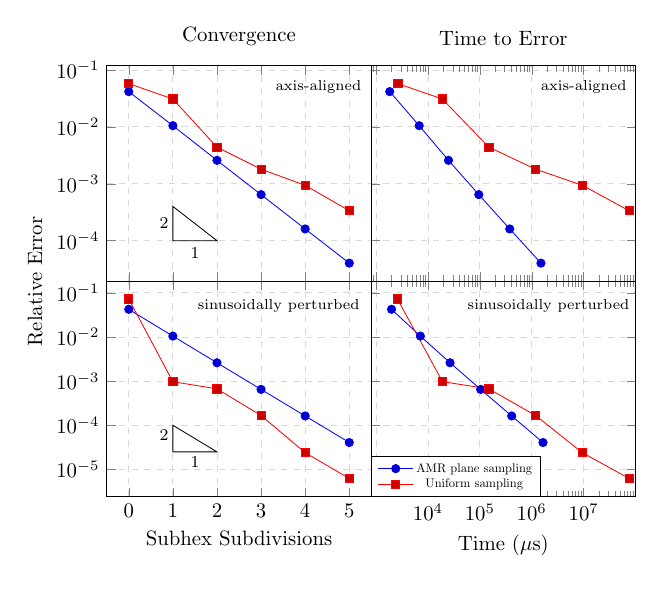}
\caption{Convergence history (left) and time to solution (right)
for a spherically capped cylinder inserted into an axis-aligned mesh (top)
and a sinusoidally perturbed mesh (bottom) for the proposed volume
fraction procedure compared to a uniform sampling procedure.}
\label{fig:pill_plots}
\end{figure}

\subsection{An Example with Poor Element Quality}
\label{ssec:results_quality}

In this section, we demonstrate that the proposed volume fraction insertion
procedure can achieve second-order convergence for a mesh with
poor quality elements. We begin by meshing the domain
$\{ \bs{x} : \bs{x} \in [-10,10]^3 \}$ with $32^3$ cubical elements
and then applying the transformation \eqref{eq:transform}
to all nodal coordinates followed by a subsequent transformation of all
nodal coordinates:
\begin{equation}
\begin{aligned}
x &= x + 2y, \\
y &= \nicefrac{3}{10} \; y, \\
z &= \nicefrac{1}{5} \; z.
\end{aligned}
\end{equation}

Figure \ref{fig:quality_mesh} illustrates the mesh and domain
obtained by this coordinate transformation. Let $Q(K^e)$ denote
the scaled Jacobian \cite{blacker1994cubit}
hexahedron quality metric, which measures a scalar deviation of an
element $K^e$ from a reference element $K^R = [-1,1]^3$ in a
geometric sense. General guidance \cite{blacker1994cubit} suggests
that this metric should remain within the range $[\nicefrac12,1]$
for `acceptable' elements. In the present mesh, 79 percent of elements
have a quality $Q(K^e) \in [0.014, 0.1]$, far below the
`acceptable' threshold.
\begin{figure}[ht!]
\begin{subfigure}{\textwidth}
\centering
\includegraphics[width=.99\linewidth]{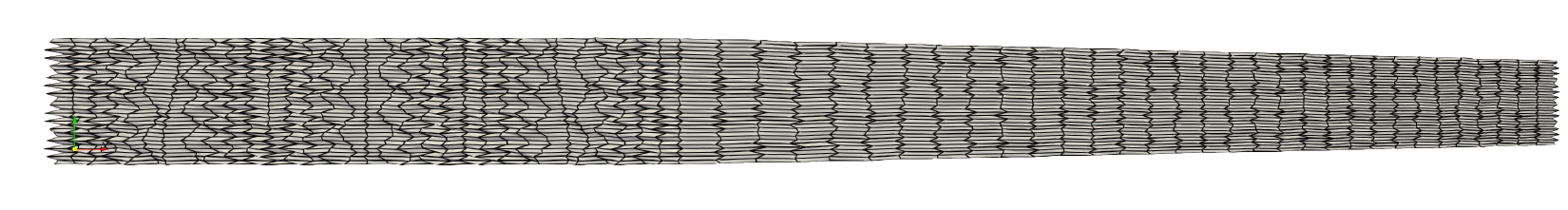}
\caption{Bottom view.}
\end{subfigure}

\begin{subfigure}{\textwidth}
\centering
\includegraphics[width=.99\linewidth]{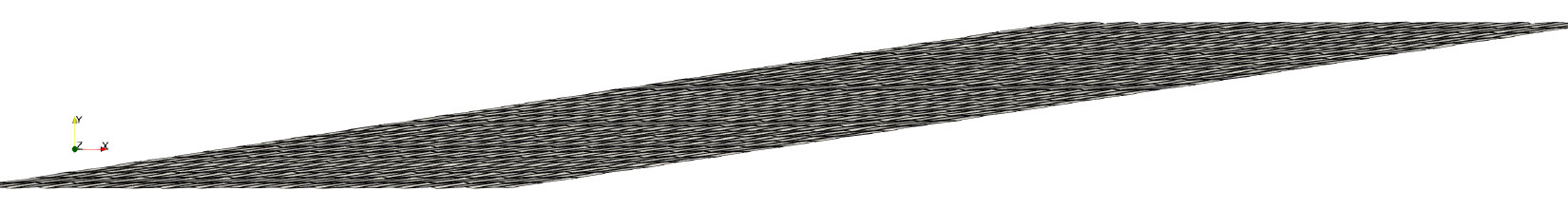}
\caption{Side view.}
\end{subfigure}

\caption{A mesh with poor quality hexahedral elements. Note that
the aspect ratio has not been exaggerated.}
\label{fig:quality_mesh}
\end{figure}

\begin{figure}[ht!]
\centering
\includegraphics[width=.4\linewidth]{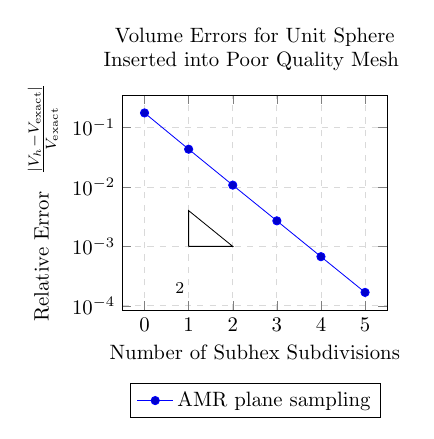}
\caption{Bottom view.}
\label{fig:quality_convergence}
\end{figure}

We measure the relative error,
$\nicefrac{|V_{\text{exact}} - V_h|}{V_{\text{exact}}}$ when inserting
a unit sphere into the mesh with poor quality hexahedra using a
maximum of $n_{\text{sub}} = \{0,1,2,3,4,5\}$ hex subdivisions.
Figure \ref{fig:quality_convergence} demonstrates that the proposed
AMR volume fraction insertion procedure is able to achieve
second-order convergence in this context. However, comparison with
Figure \ref{fig:sphere_plots}, which illustrates the convergence
history for a unit sphere inserted into meshes with better quality
hexahedra, highlights that the relative error is higher in this
context when element quality deteriorates.

We remark that this single demonstration
does not exhaustively prove the robustness of the proposed
procedure to poorly shaped background elements. It is
entirely possible that there are pathological element and
CAD geometry shapes and orientations that lead to a loss
of convergence and/or accuracy but, to date, we have not seen
such a combination. Additionally, for the majority of
physics applications that we intend to target, extremely
poor quality hexahedra are not amenable to simulation
in the first place. We propose investigating the behavior of the
proposed procedure over a wide parameterization of element quality
as an avenue for interesting future work.

\begin{figure}[ht!]
\centering
\begin{subfigure}{\textwidth}
\includegraphics[width=.99\linewidth]{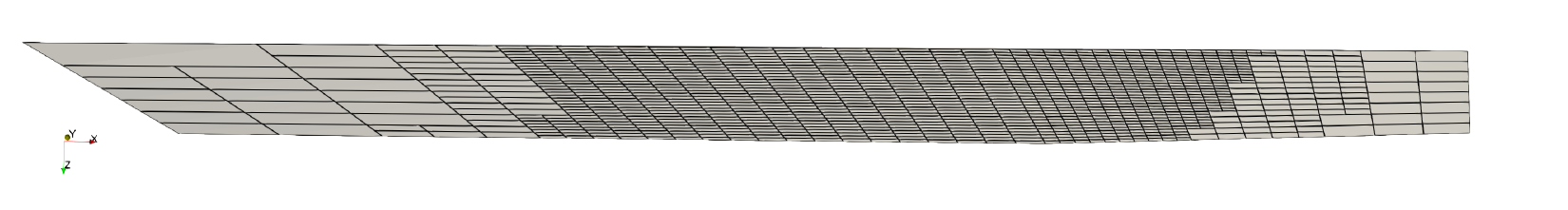}
\caption{A top down view of the element.}
\end{subfigure}

\begin{subfigure}{\textwidth}
\centering
\includegraphics[width=.99\linewidth]{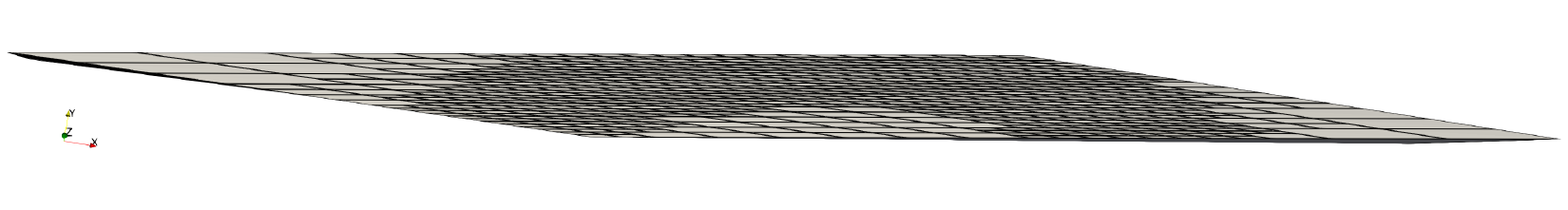}
\caption{A side view of the element.}
\end{subfigure}

\caption{An illustration of the AMR subdivision procedure for
a chosen element in the case of a unit sphere inserted into a
background mesh with poor quality hexahedra, using $n_{\text{sub}}=5$
AMR subdivisions. Note that the aspect ratio of the hexahedron
has not been exaggerated.}
\end{figure}

\subsection{A Gear into an Unstructured Mesh}
\label{ssec:results_gear}

\begin{figure}[ht!]
\centering
\begin{subfigure}{0.5\textwidth}
\centering
\includegraphics[width=.99\linewidth]{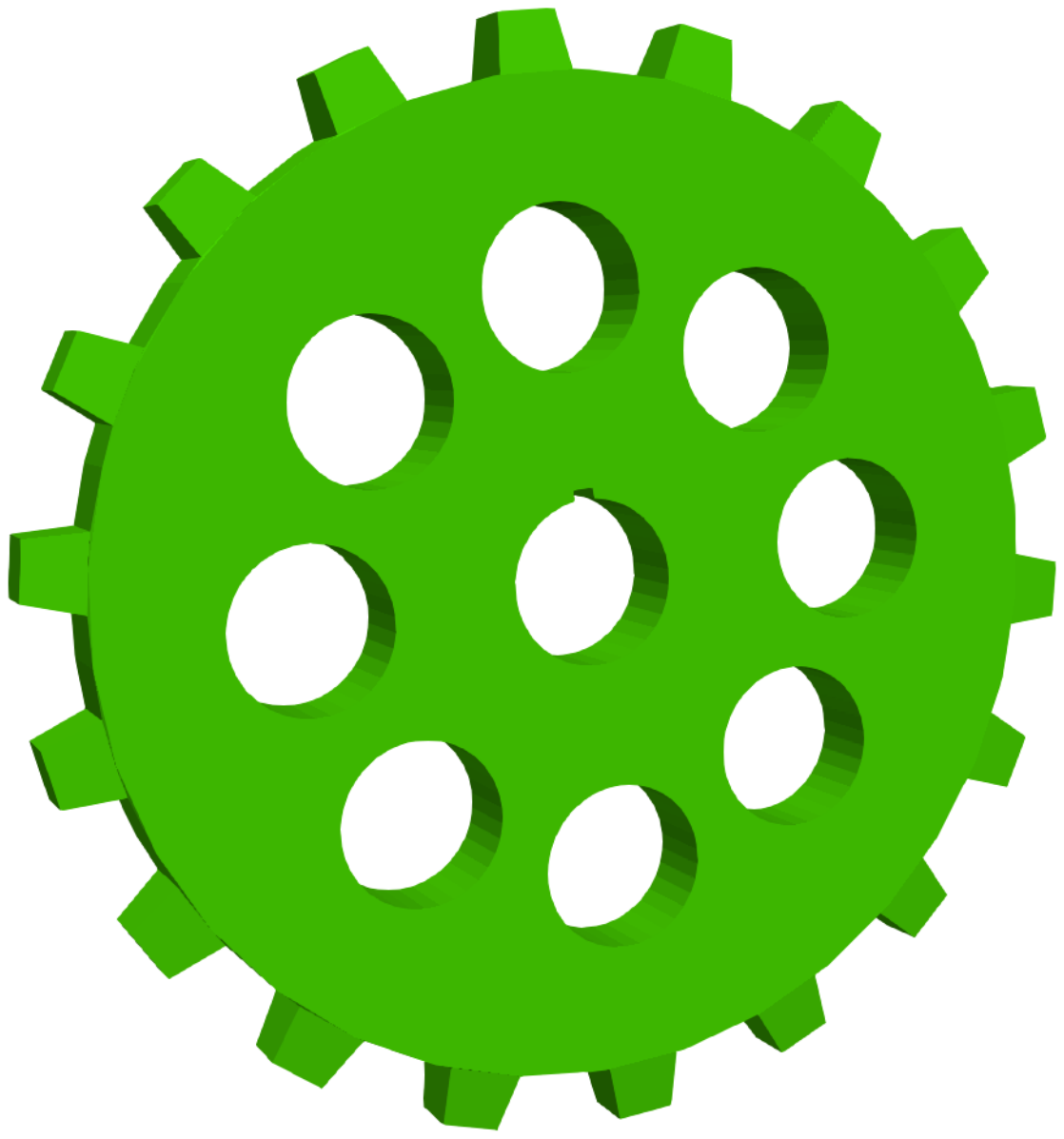}
\caption{CAD geometry.}
\end{subfigure}%
\begin{subfigure}{0.5\textwidth}
\centering
\includegraphics[width=.99\linewidth]{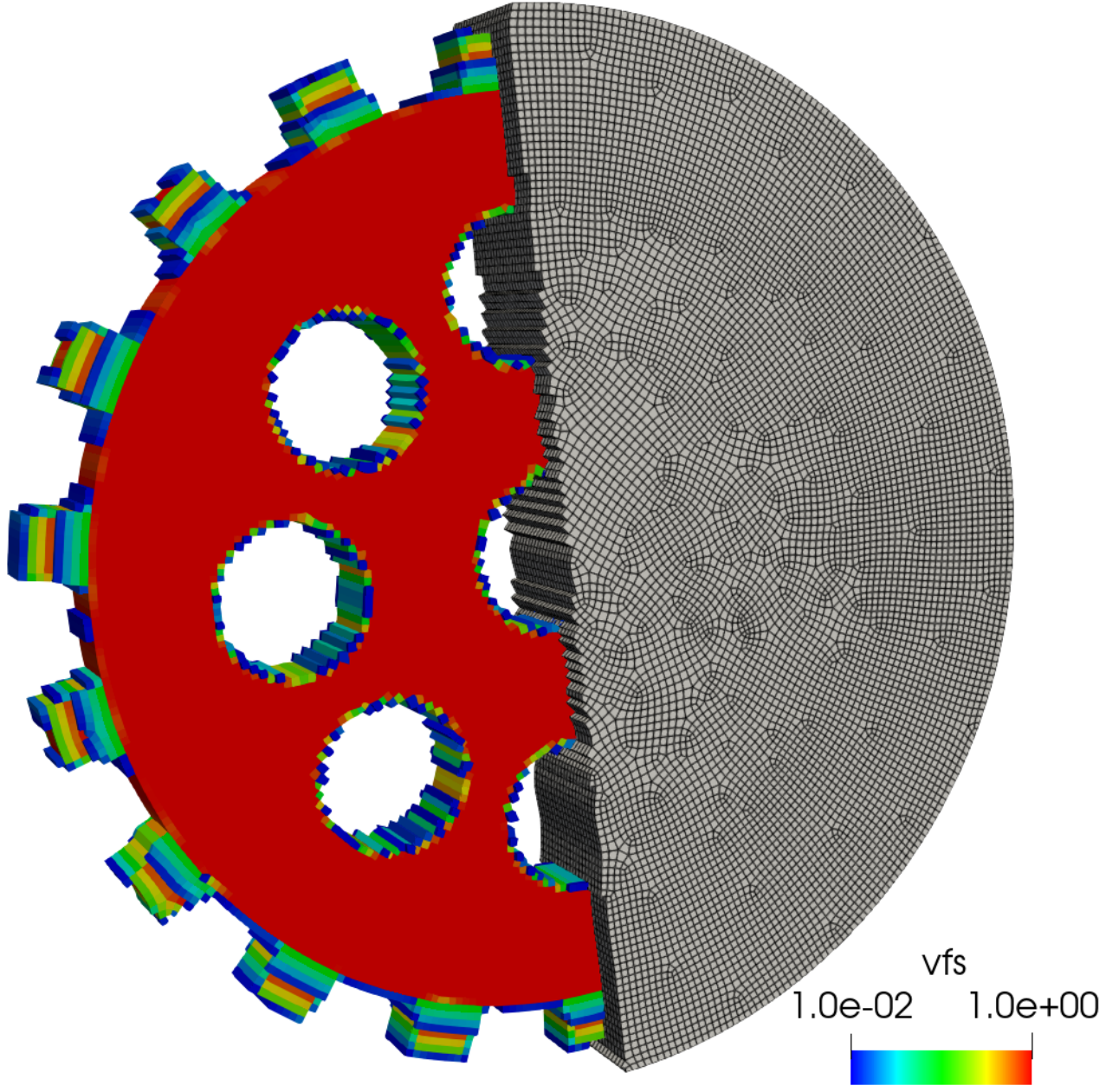}
\caption{Volume fractions and background mesh with $h=0.1$.}
\end{subfigure}
\caption{Illustrations of a gear geometry,
a background unstructured hexahedral mesh, and
the volume fractions obtained by the proposed insertion
method using two AMR subdivisions.}
\label{fig:gear}
\end{figure}

\begin{figure}[ht!]
\centering
\begin{subfigure}{0.42\textwidth}
\centering
\includegraphics[width=.99\linewidth]{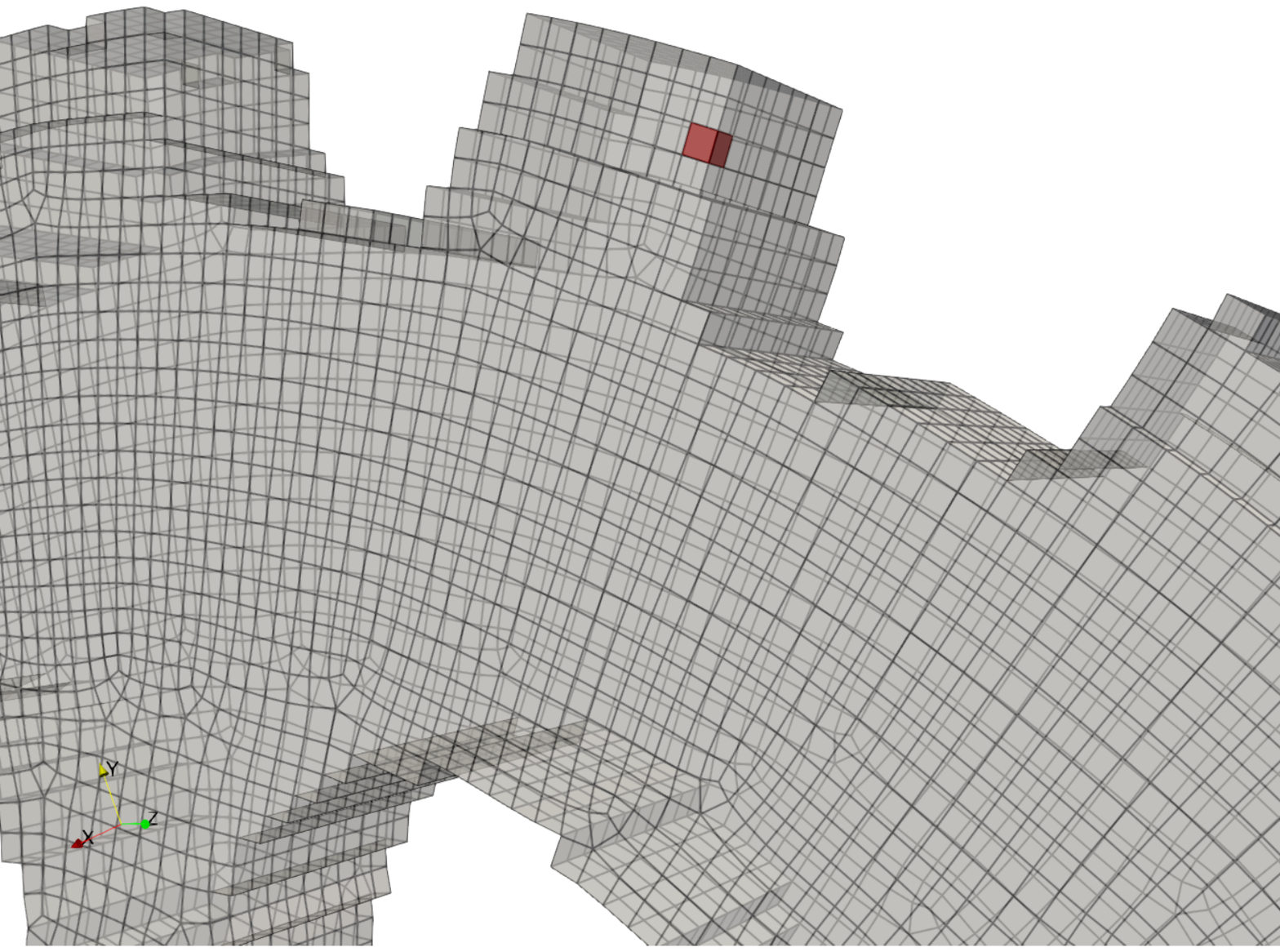}
\caption{Sample element (red).}
\end{subfigure}%
\begin{subfigure}{0.42\textwidth}
\centering
\includegraphics[width=.99\linewidth]{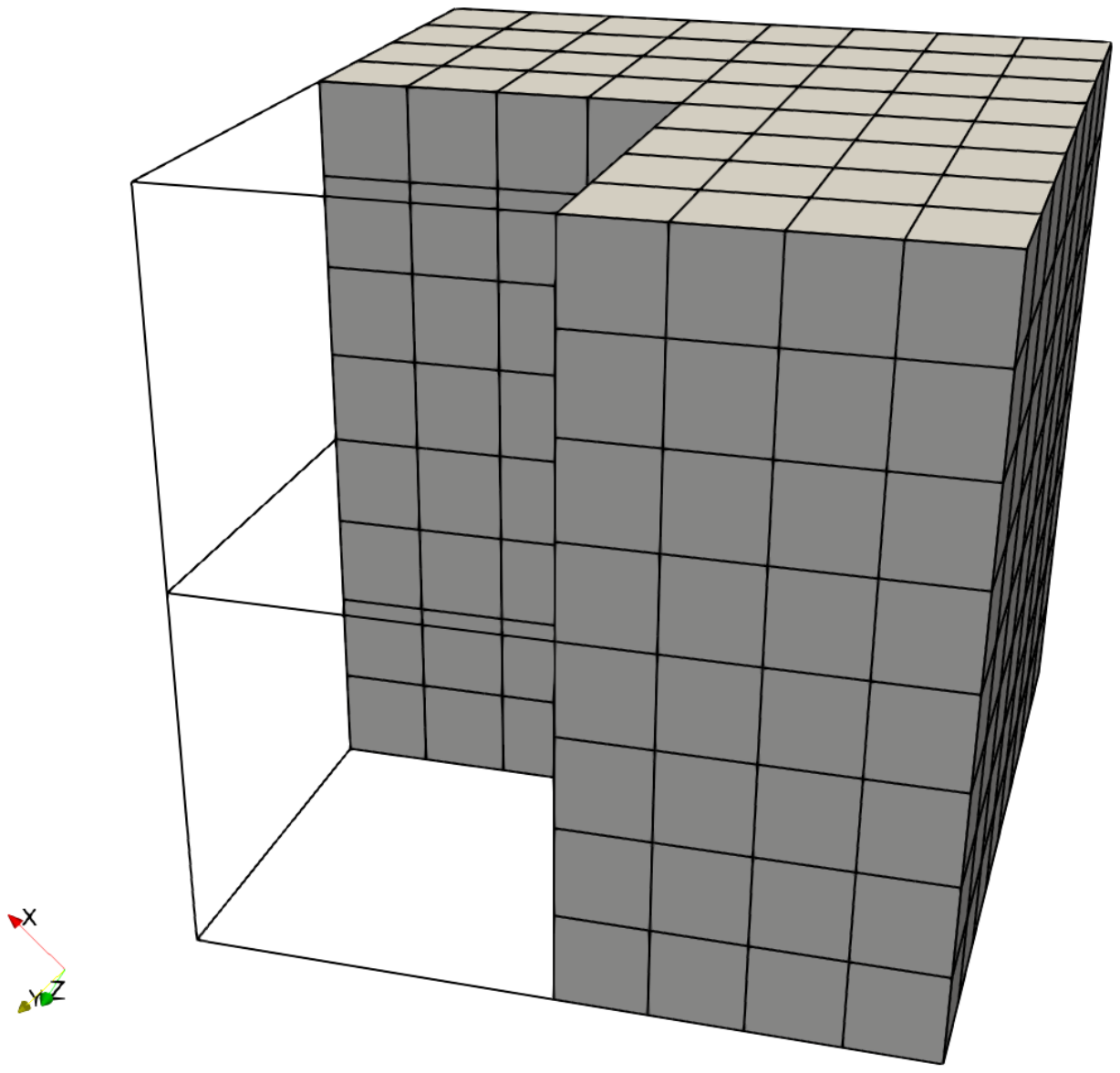}
\caption{Subhexahedra visited at the finest AMR subdivision.}
\end{subfigure}
\caption{An illustration of the AMR subdivision procedure for a
chosen element in the case of a gear geometry inserted into a
background unstructured mesh with $n_\text{sub}=3$ subdivisions.
Note that the element has been rotated to more clearly illustrate
the AMR structure.}
\label{fig:gear_elem}
\end{figure}

In this section, we investigate the accuracy and performance of the
proposed volume fraction insertion procedure applied to a more complicated
gear geometry, as shown by the left image in Figure \ref{fig:gear}.
As a first test, the CAD geometry is inserted into several meshes of
with characteristic element size $h=\{ 0.2,0.1,0.05 \}$, where the
background mesh for $h=0.1$ is shown in the right image in Figure
\ref{fig:gear}. For this test, a fixed number of AMR subdivisions
is chosen as $n_{\text{sub}} = 2$. Table \ref{tab:gear} provides information
about the AMR plane sampling procedure. The second column of Table
\ref{tab:gear} corresponds to the number of leaf nodes that are visited
in the $k$-d tree during the sampling procedure, and illustrates that
only a subset of the total hexahedra intersect the CAD geometry's surface.
The third column of Table \ref{tab:gear} corresponds to the number
of subhexahedra at the finest AMR subdivision level that intersect the
CAD geometry's surface and the final column corresponds to the theoretical
speedup, $(2^{n_{\text{sub}}}) \text{total hexes} / \text{subhexes hit}$, of the AMR
sampling procedure when compared to na\"{i}vely performing the same plane sampling
approach at every possible subhexahedron in the mesh.

\begin{table}
\begin{center}
\begin{tabular}{ c | c | c | c | c }
$h$    & total hexes & hexes hit & subhexes hit & speedup  \\ \hline
0.2    & 41380       & 11180     & 578288       & 4.57x    \\ \hline
0.1    & 324440      & 60504     & 2424304      & 8.57x    \\ \hline
0.05   & 2493320     & 234490    & 9359392      & 17.05x
\end{tabular}
\end{center}
\caption{Statistics for gear CAD geometry inserted into different meshes.}
\label{tab:gear}
\end{table}

\begin{figure}[ht!]
\centering
\includegraphics[width=.4\linewidth]{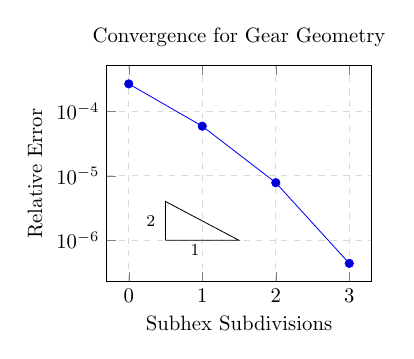}
\caption{Convergence history for the gear geometry inserted into a background
mesh of characteristic mesh size of $h=0.1$ using several maximum AMR subdivision
levels.}
\label{fig:gear_plot}
\end{figure}

As a second test, we again verify that the AMR plane sampling procedure
achieves increased accuracy as the number of AMR subdivisions
increases. In this test, the gear geometry
is inserted into a single mesh with characteristic mesh size of
$h=0.1$, as shown by the right image of Figure \ref{fig:gear}
using the number of AMR subdivisions $n_{\text{sub}} = \{0,1,2,3\}$.
Figure \ref{fig:gear_elem} illustrates the AMR subdivision of a
representative hexahedron in the background mesh when using
$n_{\text{sub}} = 3$ subdivisions.
Figure \ref{fig:gear_plot} illustrates that the relative error
$\frac{|V_{\text{exact}} - V_h|}{V_{\text{exact}}}$ decreases at
least quadratically as the number of AMR subdivisions increases.
The reference CAD volume $V_{\text{exact}}$ is reported by the
CAD kernel as $V_{\text{exact}} = 93.1818095$. In terms
of performance, the right-most data point in Figure \ref{fig:gear_plot}
took a total of $262.3$ seconds to complete.
As a comparison, the uniform sampling approach with $n_{\text{sub}} = 2$
uniform subdivisions, or $n_s = 64$ sampling points per element,
achieved a relative error of $7.305\times10^{-4}$
in $564.8$ seconds, more than twice the time required for the
most accurate AMR sampling data point. Lastly, the construction
of the $k$-d tree took $0.14$ seconds, less than
a tenth of a percent of the total volume fraction insertion runtime.

\subsection{An Example with MPI Parallelism}
\label{ssec:results_dsp}

\begin{table}
\begin{center}
\begin{tabular}{ c | c | c | c}
 & MAX & MIN & AVG \\ \hline
time (s) & 129.367 & 13.780 & 72.485
\end{tabular}
\end{center}
\caption{Timing statistics over MPI ranks for the volume insertion of
the dynamic screw pinch geometry into a 12 million element mesh
spread amongst 32 MPI ranks (as performed on a cluster with 
Intel Xeon Platinum 8260 \@2.40GHz processors).}
\label{tab:dsp_timing}
\end{table}

In this section, we investigate inserting a complex geometry using a
mesh with a distributed memory decomposition to demonstrate that
the volume fraction insertion procedure is applicable on high-performance
computing architectures. We consider a dynamic screw pinch geometry
\cite{shipley2019design} inserted into an unstructured mesh of cylinder with
12.29 million elements partitioned over 32 MPI ranks using recursive
inertial bisection \cite{ZoltanOverviewArticle2002, simon1997good}.
The left image of Figure \ref{fig:dsp} illustrates this geometry as
represented by volume fractions, where colors refer to distinct mesh
partitions. We use the proposed volume fraction insertion procedure
with $n_{\text{sub}} = 0$ adaptive subdivisions. The volume fraction
sampling proceeds by loading each mesh part onto a corresponding
single MPI rank, while the entire CAD geometry is loaded onto all
MPI ranks. Individual $k$-d trees are then constructed for each mesh
part independently (so that there is a forest of $k$-d trees) and
the remainder of the volume fraction sampling proceeds as previously
described.

Table \ref{tab:dsp_timing} illustrates the maximum, minimum, and average
runtime taken to perform the volume fraction insertion over the 32
MPI ranks. The ratio of the maximum runtime is close to an order of
magnitude larger than the minimum runtime, which highlights an important
point. The density of the geometry within the background mesh is, in general,
completely independent of the background mesh itself, which could
potentially lead to large load imbalances during the volume fraction
insertion process. It may be possible to consider partitioning the
CAD geometry itself in the same manner as the mesh, but we leave this
as an avenue for future work. Additionally, if multiple CAD geometries are
inserted in this partitioned context, such a load imbalance may eventually be
unavoidable. The right image of Figure \ref{fig:dsp}
illustrates a cutaway of an isovolume reconstructed from the inserted
volume fractions. Lastly, for most physics applications, the geometry
insertion process is a single-time upfront cost, which can also
ameliorate the negative effects of load imbalance.

\begin{figure}[ht!]
\centering
\begin{subfigure}{0.48\textwidth}
\centering
\includegraphics[width=.99\linewidth]{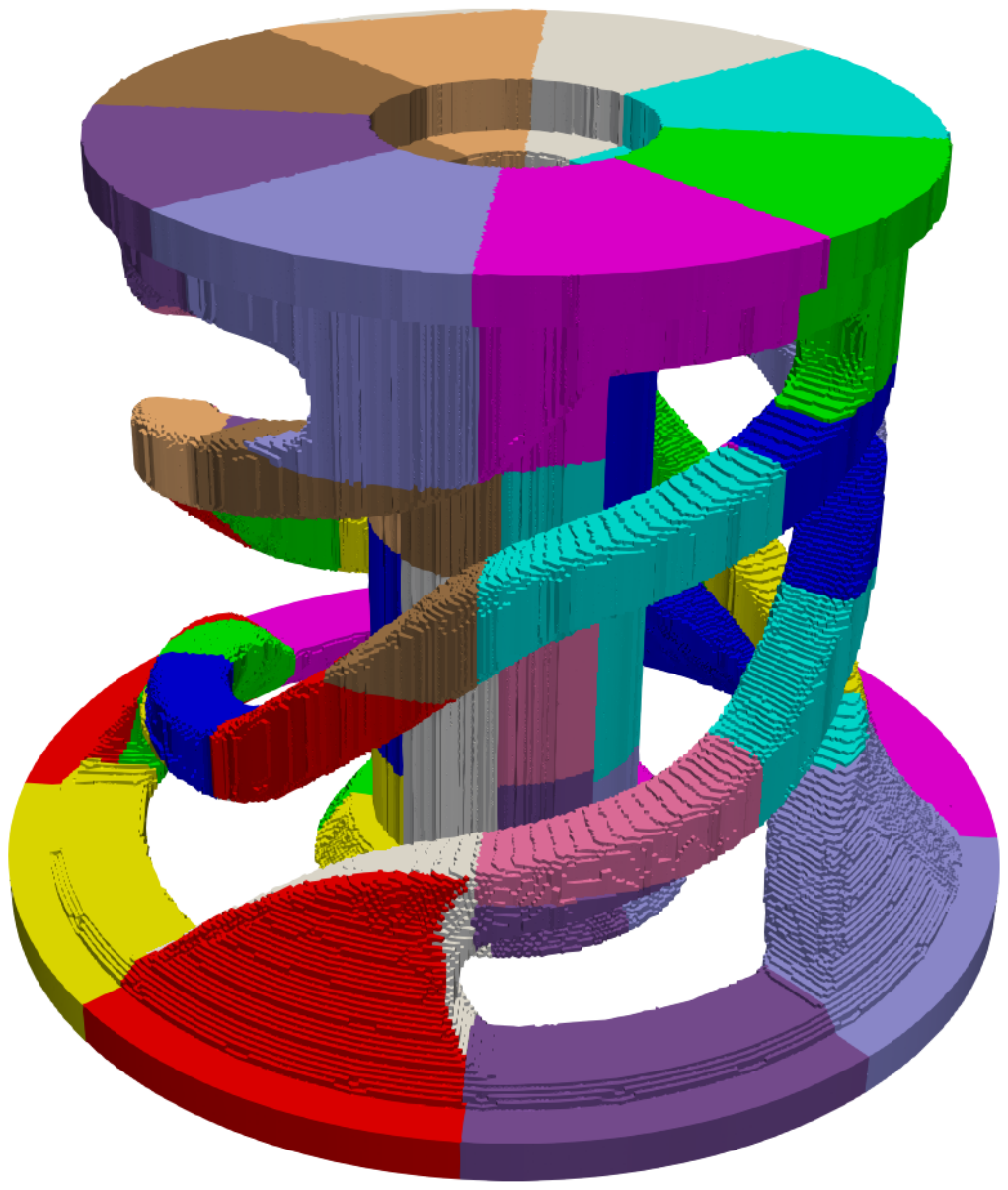}
\end{subfigure}%
\begin{subfigure}{0.48\textwidth}
\centering
\includegraphics[width=.99\linewidth]{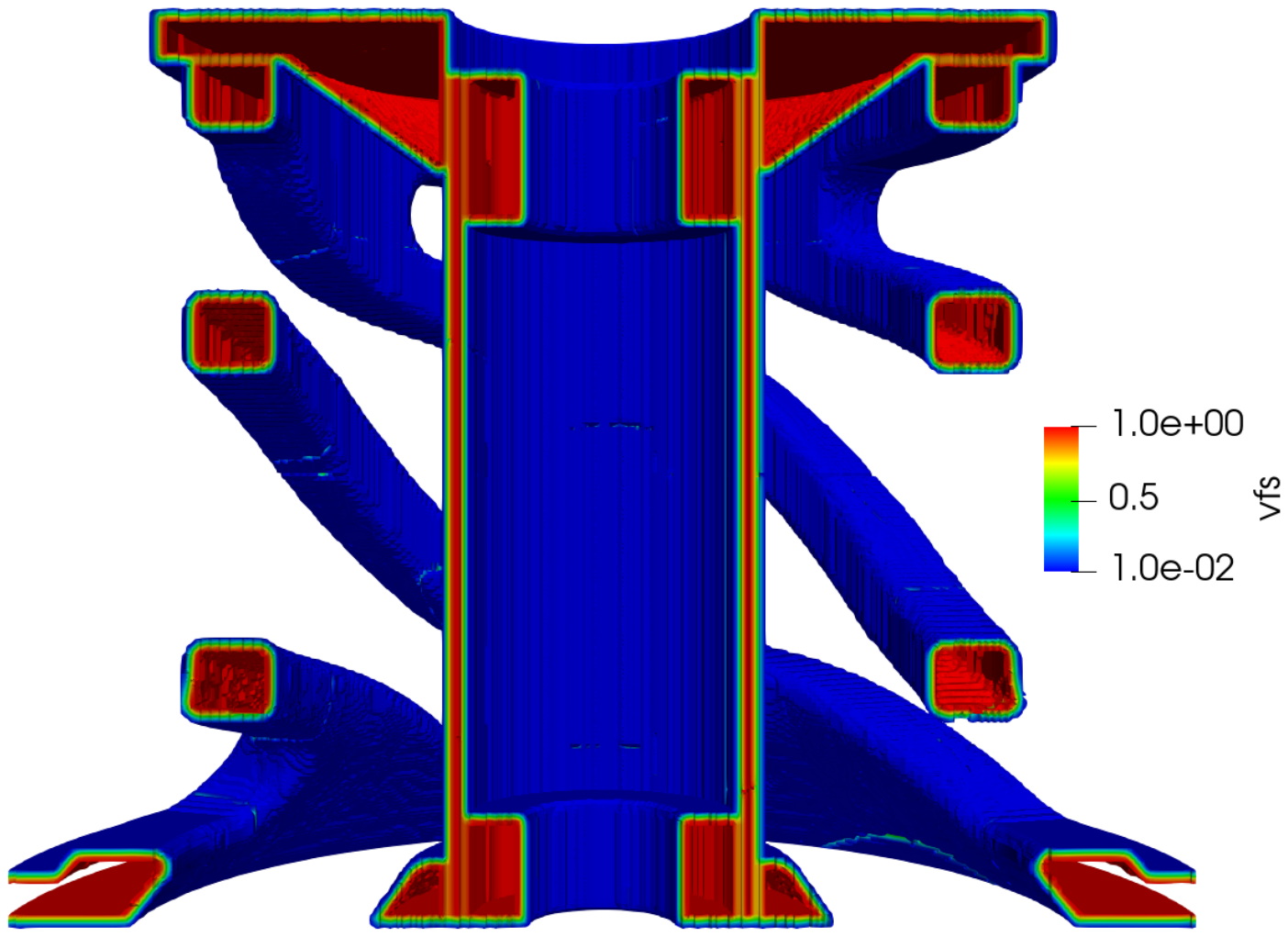}
\end{subfigure}
\caption{Left:
the dynamic screw pinch geometry as inserted into
an MPI-partitioned mesh with 32 parts where each color corresponds
to an individual mesh part.
Right:
a cutaway of an isovolume reconstruction of the inserted
dynamic screw pinch geometry into a large unstructured mesh.}
\label{fig:dsp}
\end{figure}

\section{Conclusions}
\label{sec:conclusions}

In this paper, we have developed a spatially accelerated approach to
represent three-dimensional CAD geometries by volume fractions in
background hexahedral meshes. The approach proceeds in two steps.
The first step involves assigning volume fractions for large subsets
of hexahedra that are either entirely inside or outside of the
CAD geometry through the use of a $k$-d tree. The second step involves
assigning volume fractions to elements that intersect the CAD geometry's
surface by adaptively refining such elements. At the finest AMR
subdivision level, the CAD surface is approximated linearly as a plane
and subvolume fractions are computed by intersecting a plane with
a subhexahedra. An accurate volume fraction representation is then found
by combining all subvolume fractions.

We have emphasized that this procedure requires only two queries
from a CAD kernel. First, it is necessary to determine whether or not a
spatial coordinate is inside or outside of the CAD geometry. Second,
it necessary to find the closest point on the CAD geometry's surface
to a given input spatial coordinate.
We then proved that the proposed procedure is second-order
accurate for sufficiently smooth geometries and sufficiently
refined background meshes.

To illustrate the procedure's effectiveness, we first considered
two simple geometries, a sphere and a hemispherically capped cylinder,
inserted into axis-aligned and non-axis aligned meshes. For all cases,
the proposed volume fraction insertion procedure achieved the theoretical
second-order convergence rate. Additionally, we've compared the proposed
volume fraction insertion procedure to a simpler uniform in/out sampling
procedure and demonstrated that, in most cases, the proposed procedure
achieves greater accuracy in less computational runtime. Further, we've
demonstrated that the procedure can be used reliably and robustly when
a background mesh with poor element quality is considered.
We have also demonstrated the utility of the approach for
two more realistic geometric examples.

We remark that the current work need not be limited to hexahedral
meshes. For instance, extensions to tetrahedral meshes would be trivial.
In fact, provided a bounding sphere can be computed and a uniform
refinement pattern exists for an individual element, the proposed procedure
can be applied to arbitrary element types. As an avenue for future work,
we propose applying the procedure to mixed element type meshes.
Additionally, algorithms for more sophisticated moment-of-fluid volume
representations could be developed and investigated. More
sophisticated spatial acceleration data structures could be considered
rather than the simple $k$-d tree we have considered presently.

The advantages of the newly proposed approach when compared to
existing triangulation-based methods depend on the specific
geometry, the required inserted volume accuracy, the background
mesh and its parallel decomposition. Comparisons between different
volume insertion algorithms embedded within actual engineering/scientific workflows
would be an interesting avenue for future study.

\end{document}